\documentclass[a4paper,12pt]{amsart}

\usepackage{pgf}
\usepackage{amscd,amsmath,amssymb,amstext,amsfonts,amsthm,latexsym}
\usepackage{verbatim, url}
\usepackage[english]{babel}
\usepackage[latin1]{inputenc} 
\usepackage{graphicx}

\usepackage[left=2cm,right=2cm, top=3cm, bottom=3cm]{geometry}

\usepackage[all]{xy}
\usepackage{longtable}
\usepackage{fancyhdr}
\usepackage{setspace}
 \usepackage[usenames,dvipsnames]{pstricks}
 \usepackage{epsfig}

\usepackage{hyperref}

\newtheorem{theorem}{Theorem}[section]
\newtheorem*{thm}{Theorem}  
\newtheorem{lemma}[theorem]{Lemma}

\newtheorem{proposition}[theorem]{Proposition}

\theoremstyle{remark}
\newtheorem{remark}[theorem]{Remark}

\theoremstyle{definition}

	\newcommand{\GL}{\mathrm{GL} }
																								
\title{A family of Threefolds of general type with canonical map of high degree}
\makeatletter

\@addtoreset{equation}{section}
\makeatother
\author[D. FRAPPORTI, C. GLEI\ss NER]{DAVIDE FRAPPORTI \and CHRISTIAN GLEI\ss NER}
\keywords{Threefolds  of general type, canonical map, finite group actions} 
\subjclass[2000]{14J30, 14L30, 14.01} 
\address{ \ \newline Davide Frapporti, Christian Glei\ss ner;\newline
 University of Bayreuth, Lehrstuhl Mathematik VIII; \newline
Universit\"atsstra\ss e 30, D-95447 Bayreuth, Germany}

\thanks{The authors thank S. Coughlan for inspiring 
conversations and a careful reading of the paper.
 A very special thanks  goes to R. Pignatelli 
 for the inspiring conversation at the Conference
``A Journey through Projective and Birational Geometry
Together with Marco Andreatta'' (Trento, January 2019),
where he encouraged  and helped the authors to expose	 the construction
 using an elementary approach.\\
The present work took place in the framework of the ERC Advanced grant n. 340258-TADMICAMT.\\
The first author is member of  G.N.S.A.G.A. of I.N.d.A.M.}

\date{\today}

\begin{document}

\begin{abstract} 
In this note we provide a two-dimensional family of smooth minimal threefolds of general type
with canonical map of degree 96, improving the previous known bound of 72.
\end{abstract}

\maketitle

\section*{Introduction} 
In this paper we consider smooth varieties of general type.\\
In the case of curves it is classically known that the canonical map is either an embedding or
a degree 2 map onto $\mathbb{P}^1$, the latter happens precisely when the curve is hyperelliptic.

In higher dimensions the situation is much less clear and 
it is  natural to ask: 
``Assuming that the canonical map is generically finite,
is its degree universally bounded? If so, 
what is the maximal possible value of the \textit{canonical degree}?''
In this paper by canonical degree we mean the degree of the canonical map.

For surfaces Beauville (\cite{Beau79}) showed that the canonical degree is at most 36, and  equality holds if and only if $p_g=3$, $q=0$, $K^2=36$ and the canonical system is base point free.

In the  threefold case, Hacon  \cite{Hacon_can}  established  576 as a bound for the canonical degree.
Later on in \cite{DG16_can} this bound was improved to 360, which can be achieved if and only if 
$p_g = 4$, $q_1= 2$, $\chi(\mathcal O) = -5$, $K^3 = 360$ and  the canonical system is base point free.
In  \cite{Hacon_can} Hacon also explained that if one allows terminal singularities, there is no bound for the canonical degree, presenting an infinite series of threefolds 
 with index 2 terminal singularities and arbitrarily  high canonical degree.

Beauville  and Hacon's proofs rely heavily on the  Miyaoka-Yau inequality, which cannot be used to control the canonical degree in higher dimensions.

For surfaces, as well as for Gorenstein minimal threefolds, the maximal value  can be achieved 
only  if the Miyaoka-Yau inequality becomes an equality, i.e.~for ball-quotients.
Unfortunately these are notoriously hard to handle and it  is not clear if  the bounds are sharp.
 In \cite{Yeung17} Yeung claims the existence of  a surface  realizing degree 36, but the proof seems to have a gap, as pointed out by the Mathscinet Reviewer 
 (MR Number: MR3673651).
More recently Rito \cite{Rito19} used the Borisov-Keum equations of a fake projective plane to costruct a surface having
 canonical map of degree 36.

Leaving ball-quotients aside, one can still try to construct examples having high canonical degree.
For surfaces, the highest canonical degrees obtained so far are: 16 due to Persson (\cite{Persson78}, 
double covering of a Campedelli surface),
24 due to Rito (\cite{Rito17_24}, abelian covering of a blow-up of $\mathbb P^2$),  and 32  due to  the
second author  together with Pignatelli and Rito (\cite{GPR18}, product-quotient surface).
In dimension 3 the situation   is less established.
The highest known values of the canonical degree are 32, 48, 64 and  72, which can be achieved  simply by taking the product $C\times S$, where  $C$  is a hyperelliptic curve and $S$ is one of the above mentioned surfaces.
Other examples with canonical degree 32 and 64 were  constructed in \cite{Cai08}.

In this paper we set a new record for canonical degree of a smooth  threefold:

\begin{thm}[Theorem \ref{MainThm}]
There exists a two dimensional family of smooth threefolds $X$ with 
canonical degree 96 and 
whose canonical image is a quadric.
\end{thm}

The family lives in a 6-dimensional family of threefolds isogenous to a product.
Threefolds isogenous to a product are a special case of 
\textit{product-quotient} varieties  (see \cite{Cat00,FG16,Glei17}), i.e.  varieties 
birational to  the quotient of the product of  smooth projective curves
by the action of a finite group.

In recent years the intensive work on 
product-quotient varieties has produced several interesting examples,
e.g.  the  mentioned example of a surface of general type with canonical map of degree 32;
new topological types for surface of general type, in particular a family of surfaces of general type with $K^2=7$, $p_g=q=2$ (see \cite{CF18} and the references therein);
 and recently the first examples of rigid but not infinitesimally rigid compact complex manifolds (\cite{BP18}).  For other interesting  applications  see   \cite{Cat_degree}, \cite{FGP18},  \cite{GRR18}, \cite{LP16} and \cite{LP18}.

We point out that we originally found  this 2-dimensional family using the technology developed in 
the papers mentioned above, together with the
theory  of abelian coverings (see  \cite{Liedtke03,par91}).
However, we decided not to use the language of product-quotients here, 
but to give a simple and self-contained description using  explicit equations instead.

The paper is organized as follows: in Section \ref{construction} we construct the 6-dimensional family 
and in Section \ref{TCM} we prove the main Theorem.

 \section{The construction}\label{construction}

Let $C_{a,b}$ be the curve in  $\mathbb{P}^4$ defined by the equations
\begin{equation}\label{equationCab}
x_2^2=x_0^2-x_1^2\,, \quad x_3^2=x_0^2-ax_1^2\,, \quad x_4^2=x_0^2-bx_1^2\,,
\end{equation}
where  $a,b\in\mathbb C \setminus\{0,1\}$, $a\neq b$.
The curve $C_{a,b}$  is then  a smooth canonical curve  of genus 5.
We note that $C_{a,b}$ is  invariant under the   $(\mathbb Z_2)^4 $-action on $\mathbb P^4$ given by
\[e_i\colon  x_i \mapsto -x_i \qquad x_j\mapsto x_j  \  (i\neq j)  \] 
where $\{e_1,\ldots, e_4\}$ is the standard basis of $(\mathbb Z_2)^4 $
and  $e_0:=e_1+e_2+e_3+e_4$.

The map  $\pi: C_{a,b}\to \mathbb P^1$, $\pi(x_0:x_1:x_2:x_3:x_4)= (x_0^2: x_1^2)$ is a covering of the projective line branched in 5 points: 
\[p_0=[0:1]\,, \quad p_1=[1:0]\,, \quad p_2=[1:1]\,, \quad p_3=[a:1]\,, \ \text{ and } \ \  p_4=[b:1] \,.\]
Since this map has degree 16 and  is $(\mathbb Z_2)^4$-equivariant, it is the quotient map,
and the stabilizer of a ramification point in $\pi^{-1}(p_i)$ is the group of order 2 generated by $e_i$.

We consider now the threefold  $T:=C_{a_1,b_1} \times C_{a_2,b_2}\times C_{a_3,b_3}
\subset \mathbb P^4_{\textbf x} \times \mathbb P^4_{\textbf{y}} \times\mathbb P^4 _{\textbf z} $ 
with the ``twisted'' $(\mathbb Z_2)^4 $ action defined by 
\[e_i( \textbf x,\textbf y, \textbf z):= (e_i \cdot \textbf x, (Ae_i) \cdot \textbf y, (A^2e_i)\cdot \textbf z)\,,\]
where $\textbf x:=(x_0:x_1:x_2:x_3:x_4)$ (similarly for \textbf y and  \textbf z) and $A  \in \GL (4,\mathbb Z_2)$ is the matrix
\[A:= \left( 
\begin{array}{cccc}
0&1&0&1\\
0&1&1&1\\
1&1&1&0\\
1&0&1&0
\end{array}
\right)
\qquad \text{ and } \qquad 
A^2= \left( 
\begin{array}{cccc}
1&1&0&1\\
0&0&1&1\\
1&1&0&0\\
1&0&1&1
\end{array}
\right)\,.
\]

\begin{remark}\label{rmk11}
1) $A$  is chosen so that 
\[\{e_0,\ldots,e_4, Ae_0,\ldots, Ae_4, A^2e_0,\ldots, A^2e_4\}=(\mathbb Z_2)^4 \setminus \{0\}\,.\]
2) $A$ has order 3 and satisfies $I+A+A^2=0$.
\end{remark}

The rest of the paper is devoted to the study of the canonical map
of the quotient threefold 
\begin{equation}\label{DefX}
X:=(C_{a_1,b_1} \times C_{a_2,b_2}\times C_{a_3,b_3})/(\mathbb Z_2)^4\,.
\end{equation}

\begin{lemma}
 The threefold $X$ is smooth of general type, with ample canonical class $K_X$ and 
invariants
\[
\chi(\mathcal O_X)=-4 \,, \qquad 
K_X^3 = -48\chi(\mathcal O_X)= 192\,. 
\]
\end{lemma}
\begin{proof}
The $e_i$ are the unique non-trivial elements of $(\mathbb Z_2)^4$ having fixed points on the curve.
According to  Remark \ref{rmk11}, $Ae_j \neq e_k $ for all $j,k$, therefore 
 the $(\mathbb Z_2)^4$-action on the product $T$ is free and $X$  is smooth.

Since $T$ is of general type, with ample canonical class and
the action is free, the quotient $X$  is also  of general type with ample canonical class. 
Moreover, by the freeness of the action it follows
\begin{eqnarray*}
16\chi(\mathcal O_X)=  \chi(\mathcal O_T)= \prod_{i=1}^3 (1-g(C_{a_i,b_i}))= -4^3 \\
16K_X^3=  K^3_T= 6\prod_{i=1}^3 (2g(C_{a_i,b_i})-2)= -48\chi(\mathcal{ O}_T) 
\end{eqnarray*} 
\end{proof}

\begin{remark}
Since $C_{a_i,b_i}$ is a canonical curve, the restriction 
 $H^0(\mathbb{P}^4, \mathcal O (1)) \longrightarrow H^0(C_{a_i,b_i},K_{C_{a_i,b_i}})$
 is an isomorphism. Therefore
 \begin{eqnarray*}
 H^0(C_{a_1,b_1}, K_{C_{a_1,b_1}} )&=& \langle x_0,\ x_1, \ x_2, \ x_3, \ x_4\rangle\,,\\
  H^0(C_{a_2,b_2}, K_{C_{a_2,b_3}} )&=& \langle y_0,\ y_1, \ y_2, \ y_3, \ y_4\rangle\,,\\
   H^0(C_{a_3,b_3}, K_{C_{a_2,b_3}} )&=& \langle z_0,\ z_1, \ z_2, \ z_3, \ z_4\rangle\,.\\
 \end{eqnarray*}
\end{remark}
\vspace{-.5cm}

\begin{lemma}\label{Hodge_num}
The canonical system of $X$ is spanned by 
\[s_i:=x_iy_iz_i\,, \qquad i=0,\ldots, 4\,.\]
In particular $p_g(X)=5$.
\end{lemma}

\begin{proof}
By K\"unneth's formula we have the following isomorphism
\[
H^0(K_T)\cong
H^{0}(K_{C_{a_1,b_1}} ) \otimes H^0(K_{C_{a_2,b_2}}) \otimes H^0(K_{C_{a_3,b_3}})=
\langle x_i y_jz_k\rangle_{i,j,k}.
\]

The $(\mathbb Z_2)^4$-action  on the product induces an  action  on $H^0(K_T)$ via pull-back:
\[e_\alpha^*(x_iy_jz_k) = e_\alpha^*(x_i) \cdot (Ae_\alpha)^*(y_j) \cdot (A^2e_\alpha)^*(z_k) =(-1)^{n_\alpha(i,j,k)}   x_i y_jz_k\,,\]
where   $n_\alpha(i,j,k)=I_{i\alpha}+ A_{j\alpha}+A^2_{k\alpha}$ for $\alpha=1,2,3,4$;
and  $I_{0\alpha}:=0$, $A_{0\alpha}:=0$, $A^2_{0\alpha}:=0$  for all $\alpha$. 
By the freeness of the action it holds
\[H^0(X, K_X) =H^0(T,K_T)^{(\mathbb Z_2)^4}= \langle x_iy_jz_k \mid n_\alpha(i,j,k)=0\,, \forall \alpha
\rangle\,,\]
and we conclude
\[ H^0(X, K_X) = \langle x_0y_0z_0, \ x_1y_1z_1,\ x_2y_2z_2,\ x_3y_3z_3,\ x_4y_4z_4 \rangle\,,\]
since 
$I_{i\alpha}+ A_{j\alpha}+A^2_{k\alpha}=0$ for all $\alpha$ if and only if
$i=j=k$ (cf.~Remark \ref{rmk11}).
\end{proof}
\begin{remark}
In a similar way one can determine all the other Hodge numbers of $X$:
\[q_2(X)=0\,, \quad q_1(X)=0\,, \quad h^{1,1}(X)=3\,, \quad h^{2,1}(X)=15\,.\]
\end{remark}

\section{The Canonical Map}\label{TCM}

\begin{theorem}
The canonical system $|K_X|$ is base point free, and the canonical image is a hypersurface in $\mathbb P^4$.
\end{theorem}

\begin{proof}
It follows from \eqref{equationCab} and the conditions on the parameters $a_k$ and $b_k$ that $x_i$, $x_j$ (resp. $y_i$, $y_j$ and $z_i$, $z_j$)   with $i\neq j $ cannot vanish simultaneously on $C_{a_1,b_1}$ (resp. $C_{a_2,b_2}$ and $C_{a_3,b_3}$).
Hence the sections $s_i=x_iy_iz_i$ have no common zeros and
 the canonical system $|K_X|$ is base point free.

Since $K_X$ is ample, the image of the canonical map  is a threefold, otherwise there would exist a curve $C\subset X$ with $K_X. C=0$.
\end{proof}

Note that the construction of $X$ depends on 6 parameters $a_1,b_1, a_2,b_2,a_3,b_3$.
We now  show that there is a 2-dimensional subfamily
whose elements have canonical degree 96.
Since
\[(\deg \varphi_{K_X}) \cdot \deg(\varphi_{K_X}(X))= K_X^3= 192= 2\cdot 96\,,\]
and the $\varphi_{K_X}(X)$ is non-degenerate,
we observe that $96$ is the maximal possible canonical degree.
This maximum is  achieved if and only if the image is a quadric in $\mathbb P^4$, i.e. the sections $s_i$ satisfy a  quadratic relation.

\begin{proposition}\label{MainProp}

The sections $s_i$ satisfy a non-trivial quadratic relation 
of the form
\begin{equation}\label{quadric}
 \lambda_0 s_0^2 + \lambda_1 s_1^2 + \lambda_2 s_2^2 + \lambda_3 s_3^2 + \lambda_4 s_4^2 =0
 \end{equation}
if and only if $a_1=a_2=a_3$ and $b_1=b_2=b_3$.
\end{proposition}

\begin{proof}

Using the equations of the curves \eqref{equationCab}, the existence of a non-trivial quadratic relation as in \eqref{quadric}  is equivalent to vanishing of the following expression
\[\begin{array}{l} \lambda_0 x_0^2y_0^2z_0^2+
\lambda_1   x_1^2y_1^2z_1^2+
 \lambda_2 (x_0^2-x_1^2)(y_0^2-y_1^2)(z_0^2-z_1^2)+\\
  \lambda_3  (x_0^2-a_1x_1^2)(y_0^2-a_2y_1^2)(z_0^2-a_3z_1^2)+
   \lambda_4 (x_0^2-b_1x_1^2)(y_0^2-b_2y_1^2)(z_0^2-b_3z_1^2)
   \end{array}
\]
for some $0\neq (\lambda_0, \ldots, \lambda_4)\in \mathbb C^5$.

 In other words, there exists a non-trivial quadratic relation as in \eqref{quadric} if and only if the following system of linear  equations has a non-trivial solution:

 {\small
\[\left(
\begin{array}{ccccc}
   1& 0& 1& 1&1\\
   0&1&  -1& -a_1a_2a_3 & -b_1b_2b_3 \\
 0 &  0& 1&  a_1 & b_1\\
 0 &  0& 1&  a_2 &b_2\\
 0 &  0& 1&  a_3 & b_3\\
0 &  0& 1&  a_1a_2 & b_1b_2\\
 0 &  0& 1&  a_2a_3& b_2b_3\\
 0 &  0& 1&  a_1a_3& b_1b_3\\
   \end{array}
    \right)\cdot 
    \left(
\begin{array}{c}
    \lambda_0\\  \lambda_1\\  \lambda_2\\
    \lambda_3\\   \lambda_4\\
       \end{array}
    \right)=0\,.
\]}

This means that the matrix has rank at most 4.
Using elementary matrix transformations under using the conditions $a_i,b_i \notin \{0,1\}$ and $a_i\neq b_i$, the matrix becomes

{\small
\[\left(
\begin{array}{ccccc}
   1& 0& 1& 1&1\\
   0&1&  -1& -a_1a_2a_3 & -b_1b_2b_3 \\
 0 &  0& 1&  a_1 & b_1\\
0 &  0& 0&  1 &\dfrac{ b_1(b_2-1)}{a_1(a_2-1)}\\
 0 &  0& 0&  0&(b_3-1)(a_2-1)-(a_3-1) (b_2-1)\\
  0 &  0& 0&  0 &(b_2-b_1)a_1(a_2-1) -(a_2-a_1) b_1(b_2-1)\\
 0 &  0& 0&  0& (b_3-b_1)a_1(a_2-1)- (a_3-a_1 ) b_1(b_2-1)\\
 0 &  0& 0& 0 & (b_2b_3-b_1)a_1(a_2-1)-(a_2a_3-a_1) b_1(b_2-1)\\
   \end{array}
    \right)
\]}

which has rank 4 if and only if the following equations hold:

{\small
\begin{equation}\label{4eq}
\left\{\begin{array}{rl}
i)&(b_3-1)(a_2-1)=(a_3-1) (b_2-1)\\
ii)&(b_2-b_1)a_1(a_2-1)=(a_2-a_1) b_1(b_2-1)\\
iii)&(b_3-b_1)a_1(a_2-1)=(a_3-a_1 ) b_1(b_2-1)\\
iv)& (b_2b_3-b_1)a_1(a_2-1)=(a_2a_3-a_1) b_1(b_2-1)\\
\end{array}
\right.\end{equation}
}

We claim that \eqref{4eq}  is  equivalent to
\begin{equation}\label{2eq}
a_1=a_2=a_3 \qquad \makebox{and} \qquad b_1=b_2=b_3\,,
\end{equation}
 under  the conditions on the $a_i$'s and $b_i$'s.
Clearly, it suffices to prove that a solution of \eqref{4eq} is also a solution of \eqref{2eq}.
 We distinguish two cases:

\underline{Case $b_1=b_2$.}
\noindent
Equation $ii)$ implies $a_1=a_2$. We resolve $i)$ after $a_3$ and substitute it in 
$iii)$:
\[
(b_3-b_2)a_2(a_2-1)= \frac{(a_2-1)(b_3-b_2)}{(b_2-1)}b_2 (b_2-1). 
\]
The factors $(a_2-1)$ and $(b_2-1)$ cancel out and the equation becomes 
\[
(b_3-b_2)(a_2-b_2)=0. 
\]
Since $a_2 \neq b_2$, we conclude $b_2=b_3$. This implies $a_1=a_3$ using $iii)$ again.

\underline{Case $b_1 \neq b_2$.} By $ii)$ we have $a_1 \neq a_2$ and we are allowed to divide $iv)$ and $iii)$ by $ii)$:
\[
\frac{b_2b_3-b_1}{b_2-b_1}= \frac{a_2a_3-a_1}{a_2-a_1} 
 \qquad \makebox{and} \qquad 
\frac{b_3-b_1}{b_2-b_1}= \frac{a_3-a_1}{a_2-a_1}.
\]
Furthermore we can assume that $b_3 \neq b_1$, since $b_3 = b_1$ would imply $b_1=b_2=b_3$ similarly to the  previous case. We rewrite the fractions above as 
\[
b_2b_3(a_2-a_1)-b_1(a_2-a_1)= a_2a_3(b_2-b_1) -a_1(b_2-b_1) 
\]
and 
\[
b_3(a_2-a_1)-b_1(a_2-a_1)= a_3(b_2-b_1) -a_1(b_2-b_1)
\]
and subtract them: 
\begin{equation}\label{st}
(b_2b_3-b_3)(a_2-a_1)=(a_2a_3-a_3)(b_2-b_1) .
\end{equation}
Next we resolve $ii)$ after $(a_2-a_1)$ and substitute in \eqref{st}. This  simplifies to $b_3a_1=a_3b_1$. 
We use this equation to rewrite $iii)$ as
\[
(b_3-b_1)a_1(a_2-b_2)=0. 
\]
Since $a_2 \neq 0$ and $a_2 \neq b_2$, we conclude that $b_1=b_3$, a contradiction. 
\end{proof}

As consequence we get: 

\begin{theorem} \label{MainThm}
There exists a two dimensional family of smooth threefolds $X$ of general type  with 
canonical degree $96$ and 
whose canonical image is a quadric.
\end{theorem}

\bibliographystyle{alpha}

\end{document}